\documentclass[reqno]{amsproc}

\usepackage{enumitem}
\usepackage{amssymb}
\usepackage{hyperref}

\newcommand*{\mailto}[1]{\href{mailto:#1}{\nolinkurl{#1}}}
\newcommand{\arxiv}[1]{\href{http://arxiv.org/abs/#1}{arXiv:#1}}

\newtheorem{theorem}{Theorem}[section]

\newtheorem{proposition}[theorem]{Proposition}
\newtheorem{corollary}[theorem]{Corollary}
\newtheorem{hypothesis}[theorem]{Hypothesis}
\newtheorem{remark}[theorem]{Remark}

\newcommand{\R}{{\mathbb R}}
\newcommand{\N}{{\mathbb N}}
\newcommand{\Z}{{\mathbb Z}}
\newcommand{\C}{{\mathbb C}}

\newcommand{\spr}[2]{\langle #1 , #2 \rangle}
\newcommand{\dbspr}[2]{[ #1 , #2 ]}

\newcommand{\E}{\mathrm{e}}
\newcommand{\I}{\mathrm{i}}

\newcommand{\im}{\mathrm{Im}}
\newcommand{\re}{\mathrm{Re}}

\newcommand{\lam}{\lambda}

\newcommand{\dom}[1]{\mathrm{dom}(#1)}

\newcommand{\trace}[1]{\mathrm{tr}(#1)}

\newcommand{\oo}{o}

\numberwithin{equation}{section}

\begin{document}

\title[Inverse uniqueness results for weighted Dirac operators]{Inverse uniqueness results for one-dimensional weighted Dirac operators}

\author[J.\ Eckhardt]{Jonathan Eckhardt}
\address{Institut Mittag-Leffler\\
Aurav\"agen 17\\ SE-182 60 Djursholm\\ Sweden}
\email{\mailto{jonathaneckhardt@aon.at}}

\author[A.\ Kostenko]{Aleksey Kostenko}
\address{Faculty of Mathematics\\ University of Vienna\\
Oskar-Morgenstern-Platz 1\\ 1090 Wien\\ Austria}
\email{\mailto{duzer80@gmail.com};\mailto{Oleksiy.Kostenko@univie.ac.at}}

\author[G.\ Teschl]{Gerald Teschl}
\address{Faculty of Mathematics\\ University of Vienna\\
Oskar-Morgenstern-Platz 1\\ 1090 Wien\\ Austria\\ and International
Erwin Schr\"odinger
Institute for Mathematical Physics\\ Boltzmanngasse 9\\ 1090 Wien\\ Austria}
\email{\mailto{Gerald.Teschl@univie.ac.at}}
\urladdr{\url{http://www.mat.univie.ac.at/~gerald/}}

\dedicatory{To Vladimir Aleksandrovich Marchenko with deep admiration}
\thanks{in "Spectral Theory and Differential Equations: V.A. Marchenko 90th Anniversary Collection", E. Khruslov, L. Pastur and D. Shepelsky (eds), 117--133, Advances in the Mathematical Sciences {\bf 233}, Amer. Math. Soc., Providence, 2014.}
\thanks{{\it Research supported by the Austrian Science Fund (FWF) under Grant No.\ Y330 and M1309 as well as by the AXA Research Fund under the Mittag-Leffler Fellowship Project}}


\begin{abstract}
Given a one-dimensional weighted Dirac operator we can define a spectral measure
by virtue of singular Weyl--Titchmarsh--Kodaira theory. Using the theory of de Branges spaces
we show that  the spectral measure uniquely determines the Dirac operator up to a gauge transformation.
Our result applies in particular to radial Dirac operators and extends the classical results for
Dirac operators with one regular endpoint. Moreover, our result also improves the currently
known results for canonical (Hamiltonian) systems. If one endpoint is in the limit circle case, we also establish
corresponding two-spectra results.
\end{abstract}

\maketitle

\section{Introduction}

In this paper we are concerned with inverse uniqueness results for one-di\-men\-sion\-al weighted Dirac operators.
Such operators include the case of one-dimensional Dirac operators which play an important role as a toy model
in relativistic quantum mechanics \cite{baev,th} as well as canonical (Hamiltonian) systems which are of independent interest.
Moreover, such operators also arise from the three dimensional Dirac equation with an (e.g.) radially symmetric potential.
In this latter case both endpoints of the resulting  radial operator will be singular and classical Weyl--Titchmarsh--Kodaira
theory leads to a definition of two by two Weyl--Titchmarsh matrices. In the case of one-dimensional spherical Schr\"odinger
operators this has led to the development of a singular version of Weyl--Titchmarsh--Kodaira theory which allows to introduce
a scalar function, the so-called  singular  Weyl function; see \cite{fl,fll,gz,kst,kst2,kst3,kt,kt2,kl}
and the references therein. Moreover, in \cite{je,je2,egnt2} this theory has been combined with the theory of de Branges  to obtain new powerful
inverse uniqueness results for Sturm--Liouville equations which have important applications to (e.g.) the Camassa--Holm equation \cite{et}.
The basic theory in the case of one-dimensional Dirac operators has been established in \cite{bekt} (see also \cite{ekt}) and the aim of the present
paper is to combine this basic theory with the theory of de Branges spaces extending the aforementioned results \cite{je,je2} to the case of Dirac
operators. As our main result we will show that the spectral measure uniquely determines the Dirac operator up to a gauge transformation
extending the classical results for the case of a regular endpoint. If one endpoint is in the limit circle case, such that we can vary the
boundary condition at this endpoint,  then we also establish corresponding two-spectra results. 

We apply our findings to the case of radial Dirac operators which have attracted significant interest recently \cite{ahm, ahm2, ahm3, ser}.
In particular, we establish a Borg--Marchenko \cite{borg,mar} result for this case extending the results from \cite{clge}. We also
extend some of the results for canonical systems from \cite{ww}.

For closely related research we also refer to \cite{egnt,egnt2,et2,gkm,sak} as well as the recent monograph \cite{sakb}.

\section{Dirac operators with strongly singular coefficients}\label{s2}

Let $(a,b)$ be a bounded or unbounded interval and consider the differential expression $\tau$ given by  
\begin{align}\label{eqnDiffExp}
 \tau f(x) = R(x)^{-1} \left( J f'(x) + Q(x) f(x) \right), \quad x\in (a,b),
\end{align}
where $J$ is the symplectic matrix  
\begin{align}
 J = \begin{pmatrix} 0 & -1 \\ 1 & 0 \end{pmatrix}
\end{align} 
and the coefficients are presumed to satisfy the following set of assumptions. 

\begin{hypothesis}\label{hyp1} Suppose that the functions $Q$, $R: (a,b)\rightarrow \R^{2\times 2}$ are measurable and satisfy the following additional conditions: 
 \begin{enumerate}[label=\emph{(}\roman*\emph{)}, leftmargin=*, widest=iiii]
  \item The functions $\|Q\|$ and $\|R\|$ are locally integrable on $(a,b)$. 
  \item The matrix $Q(x)$ is self-adjoint for almost all $x\in(a,b)$. 
  \item The matrix $R(x)$ is positive definite for almost all $x\in(a,b)$. 
 \end{enumerate}
\end{hypothesis}

\begin{remark}\label{rem:hyp}
The assumption \emph{(}iii\emph{)} is too restrictive since it excludes a number of interesting cases (for instance, Krein's strings with mass distributions having jumps provide examples of $R$ with $\det(R)=0$ on some intervals \cite[Chapter VI.8]{gk}, see also \cite{lawi, lawo, ww}). However, the results of the paper remain true under much more general assumptions on the coefficients. Namely, it suffices to require that $R$ is non-zero a.e.\ on $(a,b)$ and that the spectral problem 
\begin{align}\label{eqnSP}
 J f'(x) + Q(x) f(x) = z R(x) f(x), \quad x\in (a,b),
\end{align}
is definite in the sense of \cite[\S 2.5]{lema03}, that is, the system 
\begin{align}
J f'(x) + Q(x) f(x) =0,\quad R(x) f(x)=0,\quad x\in (a,b),
\end{align}
has only a trivial solution. Note that in the case $Q=0$ on $(a,b)$, the latter is equivalent to the fact that $R$ is of positive type, i.e., the matrix $\int_c^d R(x)dx$ is invertible for some subinterval $(c,d)\subset (a,b)$ (cf.\ \cite{gk, lema03}).

We decided to restrict our considerations to the case of positive definite $R$ since, on the one hand, our main motivation is the Dirac equation (and in this case $R= I$ on $(a,b)$). On the other hand, a rigorous definition of the operator (linear relation) associated with the spectral problem \eqref{eqnSP} in this case is lengthy (cf., e.g., \cite{ka2, lema03, ww}), however, the proofs of our main results remain the same.
\end{remark}

For given functions $f$, $g:(a,b)\rightarrow \C^2$, we introduce the Wronskian $W_x(f,g)$ via 
\begin{align}\label{eqnWdef}
 W_x(f,g) = \spr{J f(x)^\ast}{g(x)} = f_1(x) g_2(x) - f_2(x) g_1(x), \quad x\in (a,b). 
\end{align}
Here and henceforth, subscripts will denote the respective component of a vector-valued function. 
In particular, the scalar product in~\eqref{eqnWdef} is the usual one in $\C^2$. 
As long as the functions $f$ and $g$ are at least locally absolutely continuous, we have the following Lagrange identity 
\begin{align}\label{eqnLagrange}
 W_\beta(f,g) - W_\alpha(f,g) = \int_\alpha^\beta \spr{\tau f(x)^\ast}{R(x) g(x)} - \spr{f(x)^\ast}{R(x) \tau g(x)} dx
\end{align}
for all $\alpha$, $\beta\in(a,b)$ with $\alpha < \beta$; see \cite[Theorem~2.3]{wdl}. 

It is well known that the differential expression $\tau$ gives rise to self-adjoint operators in the Hilbert space $L^2((a,b);R(x)dx)$ associated with the inner product
\begin{align}
 \spr{f}{g} = \int_a^b \spr{f(x)}{R(x) g(x)} dx, \quad f,\, g\in L^2((a,b);R(x)dx).
\end{align} 
Although we use the same notation for the scalar products in $L^2((a,b);R(x)dx)$ and in $\C^2$, it will be clear from the context which one is meant. 
In the following, let $S$ be a self-adjoint realization of the differential expression $\tau$ with separated boundary conditions. 
This means that the operator $S$ is given by   
\begin{align}\label{eqnDefS}
\begin{split}
 \dom{S} = \lbrace f \in & L^2((a,b);R(x)dx) \,|\, f_1,\, f_2\in AC_{\text{loc}}(a,b), \\ & \tau f\in L^2((a,b);R(x)dx),~  W_a(f,u_a) = W_b(f,u_b) = 0 \rbrace, 
\end{split}
\end{align}
and $S f = \tau f$ for $f\in\dom{S}$. 
Hereby, the functions $u_a$, $u_b$ can for example be chosen as real-valued solutions of $\tau u=0$ which lie in $L^2((a,b);R(x)dx)$ near $a$, $b$, respectively.
In this case, the limits 
\begin{align}
 W_a(f,u_a) = \lim_{x\rightarrow a} W_x(f,u_a) \quad\text{and}\quad W_b(f,u_b) = \lim_{x\rightarrow b} W_x(f,u_b)
\end{align}
in~\eqref{eqnDefS} are guaranteed to exist in view of the Lagrange identity~\eqref{eqnLagrange}.  

More precisely, the need for boundary conditions in~\eqref{eqnDefS} actually depends on whether the limit circle case or the limit point case prevails at the respective endpoint. 
Thereby, an endpoint is said to be in the {\em limit circle} case if for every $z\in\C$, all solutions of $(\tau-z)u=0$ lie in $L^2((a,b);R(x)dx)$ near the respective endpoint.
Otherwise, for each $z\in\C$ there is a solution of $(\tau -z)u=0$ which does not lie in $L^2((a,b);R(x)dx)$ near the respective endpoint \cite[Theorem~5.6]{wdl} and the endpoint is said to be in the {\em limit point} case.
It is known that if an endpoint is in the limit point case, then the boundary condition there is superfluous, that is, the respective function $u_a$ or $u_b$ can be chosen to be identically zero. 

In this article, we will consider quite singular endpoints to the extent that the spectrum of $S$ remains simple. 
To this end, we say that a function $\Phi: \C\times(a,b)\rightarrow \C$ is a {\em real entire solution} of $(\tau-z)u=0$ if $\Phi(z,\cdot\,)$ is a solution of $(\tau-z)u=0$ for every $z\in\C$ and both of the functions $\Phi_1(\,\cdot\,,c)$ and $\Phi_2(\,\cdot\,,c)$ are real entire for one (and hence for all) $c\in(a,b)$. 
Moreover, in this case, we say that $\Phi$ lies in $\dom{S}$ near $a$ if for every $z\in\C$, the function $\Phi(z,\cdot\,)$ lies in $L^2((a,b);R(x)dx)$ near $a$ and satisfies the boundary condition at $a$ if $\tau$ is in the limit circle case there.

\begin{hypothesis}\label{hyp2}
There is a nontrivial real entire solution $\Phi$ of $(\tau-z)u=0$ which lies in $\dom{S}$ near $a$. 
\end{hypothesis}

It is known that in general, the essential spectrum of $S$ is made up of a part arising from the left endpoint and another part from the right endpoint.  
The assumption in Hypothesis~\ref{hyp2} is equivalent to presuming that there is no contribution coming from the left endpoint (see \cite[Section~2]{bekt}, \cite{gz}, \cite{kst2}). 
In this case, the spectrum of $S$ turns out to be simple as we will see below. 

Under the presumption of Hypothesis~\ref{hyp2}, one may introduce the transformation 
\begin{align}\label{eqnTrans}
 \hat{f}(z) = \int_a^b \spr{f(x)^\ast}{R(x) \Phi(z,x)} dx, \quad z\in\C,
\end{align}
for all functions $f\in L^2((a,b);R(x)dx)$ which vanish almost everywhere near the right endpoint $b$. 
Given this, it is possible to introduce a scalar spectral measure $\rho$ on $\R$ (cf.\ \cite[Section~2]{bekt}, \cite{kst2}, \cite{fll}) such that 
\begin{align}
 \int_\R |\hat{f}(\lambda)|^2 d\rho(\lambda) = \int_a^b \spr{f(x)}{R(x)f(x)} dx
\end{align} 
for all functions $f$ vanishing almost everywhere near $b$. 
Moreover, the transformation in~\eqref{eqnTrans} uniquely extends to a unitary operator $\mathcal{F}$ from $L^2((a,b);R(x)dx)$ onto $L^2(\R;d\rho)$, which maps the operator $S$ onto multiplication with the independent variable in $L^2(\R;d\rho)$. 
Hereby, the inverse of $\mathcal{F}$  is given by 
\begin{align}
 \mathcal{F}^{-1} F(x) = \int_\R \Phi(\lambda,x) F(\lambda) d\rho(\lambda), \quad x\in(a,b),
\end{align}
for functions $F\in L^2(\R;d\rho)$ with compact support. 
Because of this, the measure $\rho$ is called the spectral measure of $S$ associated with the real entire solution $\Phi$. 

One way to obtain the spectral measure $\rho$ is by introducing a singular Weyl--Titchmarsh--Kodaira function \cite{kst2}, \cite{bekt}. 
Therefore, one requires an additional real entire solution $\Theta$ of $(\tau-z)u=0$ such that $W(\Theta,\Phi)=1$. 
Given such a fundamental system $\Theta$, $\Phi$ of solutions, one may introduce the corresponding singular Weyl--Titchmarsh--Kodaira function $M$ on $\C\backslash\R$ by requiring that the solution
\begin{align}\label{eq:m_sing}
 \Psi(z,x) = \Theta(z,x) + M(z) \Phi(z,x), \quad x\in(a,b),
\end{align}
lies in $\dom{S}$ near $b$ for every $z\in\C\backslash\R$. 
With this definition, the spectral measure $\rho$ of $S$ associated with the real entire solution $\Phi$ is given by 
\begin{align}\label{eqnSMdef}
 \rho((\lambda_1,\lambda_2]) = \lim_{\delta\rightarrow 0} \lim_{\varepsilon\rightarrow 0} \frac{1}{\pi} \int_{\lambda_1+\delta}^{\lambda_2+\delta} \im(M(\lambda-\I\varepsilon)) d\lambda,
\end{align}
for all $\lambda_1$, $\lambda_2\in\R$ with $\lambda_1<\lambda_2$. 

The purpose of the present article is, to investigate to which extent a spectral measure $\rho$ determines the coefficients as well as the boundary conditions of the underlying operator $S$.
 This will in general only be possible up to a so-called Liouville transformation.
 In order to present the concept of a Liouville transformation, let $\eta$ be a locally absolutely continuous, increasing bijection from $(a,b)$ onto  some other interval $(\tilde{a}, \tilde{b})$ and $\Gamma: (a,b)\rightarrow\R^{2\times 2}$ be locally absolutely continuous such that $\det(\Gamma(x))=1$ for all $x\in(a,b)$. 
Then it is simple to check that the transformation
\begin{align}\label{eqnLiou}
 \tilde{f} \mapsto f(x) = \Gamma(x) \tilde{f}(\eta(x)), \quad x\in(a,b),
\end{align}
maps solutions of the Dirac equation $\tilde{\tau}\tilde{u}=\tilde{g}$ 
on the interval $(\tilde{a},\tilde{b})$, with the coefficients $\tilde{Q}$, $\tilde{R}$ of the differential expression $\tilde{\tau}$ given by 
\begin{align}
  \eta' \tilde{R}\circ\eta & = \Gamma^\ast R\, \Gamma,  &  \eta' \tilde{Q}\circ\eta & = \Gamma^\ast Q\, \Gamma +  \Gamma^\ast J\, \Gamma',
\end{align}
onto solutions of the Dirac equation $\tau u=g$. 
Moreover, the transformation~\eqref{eqnLiou} gives rise to a unitary operator from $L^2((\tilde{a}, \tilde{b}); \tilde{R}(x)dx)$ onto $L^2((a,b);R(x)dx)$ which maps the operator $S$ onto a self-adjoint realization $\tilde{S}$ of $\tilde{\tau}$ in $L^2((\tilde{a}, \tilde{b}); \tilde{R}(x)dx)$. 
In particular, the Liouville transform maps any real entire solution $\tilde{\Phi}$ which lies in $\dom{\tilde{S}}$ near the left endpoint to a corresponding real entire solution $\Phi$ as in Hypothesis~\ref{hyp2}. 
Since the Wronskian is invariant under this transformation, the corresponding spectral measures are identical (that is, $\rho = \tilde{\rho}$).

A particular kind of these Liouville transformations can be performed when the weight matrix $R$ is assumed to be locally absolutely continuous. 
In this case, we may choose locally absolutely continuous functions $\eta$ and $\Gamma$ on $(a,b)$ such that 
\begin{align}
 \eta'(x) = \sqrt{\det(R(x))} \quad\text{and}\quad  \Gamma(x) = \sqrt{\eta'(x) R(x)^{-1}}, \quad x\in(a,b).
\end{align} 
Now the inverse of the transformation~\eqref{eqnLiou} maps the differential expression $\tau$ to the differential expression $\tilde{\tau}$ which now has the constant weight $\tilde{R}=I$. 

Upon performing another Liouville transform, we can even normalize the trace of the potential matrix to zero. 
Therefore, we choose $\eta(x)=x$ and $\Gamma$ such that 
\begin{align}
 \varphi'(x) = \frac{\trace{Q(x)}}{2}, \quad\text{and}\quad \Gamma(x) = \E^{\varphi(x)J}, \quad x\in(a,b).
\end{align}
Then the inverse of the transformation~\eqref{eqnLiou} maps the differential expression $\tau$ to the differential expression $\tilde{\tau}$ whose potential matrix now has trace zero almost everywhere. 
Hereby also note that one has $\tilde{R}=I$ as long as initially $R=I$.

Similarly, one can also perform a Liouville transformation which reduces the potential matrix to zero. 
Therefore, we simply choose $\eta(x)=x$ and let $\Gamma$ be a solution of $J\,\Gamma' + Q\,\Gamma = 0$ with $\det(\Gamma)=1$ (hereby note that the determinant of such a solution is constant because $\trace{JQ}$ is zero almost everywhere). 
With this choice, the inverse of the transformation~\eqref{eqnLiou} maps the differential expression $\tau$ to the differential expression $\tilde{\tau}$ which now has no potential matrix $\tilde{Q}=0$. 

Finally, we may also perform a Liouville transform which normalizes the determinant of the weight matrix to one. 
To this end, we choose $\Gamma$ to be constant and $\eta$ such that 
\begin{align}
 \eta'(x) = \sqrt{\det(R(x))}, \quad x\in(a,b). 
\end{align}
In this case, the inverse of the transformation~\eqref{eqnLiou} maps the differential expression $\tau$ to the differential expression $\tilde{\tau}$ whose  weight matrix has determinant one almost everywhere. 
Also note that under this transformation, $\tilde{\tau}$ has no potential term if and only if $\tau$ has none.

\section{Dirac operators and de Branges spaces}\label{secdbs}

The proof of our inverse uniqueness result relies on de Branges' subspace ordering theorem. 
Hence in this section, we will first introduce a chain of de Branges spaces associated with a self-adjoint Dirac operator $S$ with separated boundary conditions, which satisfies Hypothesis~\ref{hyp1} and Hypothesis~\ref{hyp2}. 

To this end, we fix some point $c\in (a,b)$ and consider the entire function 
\begin{align}\label{eq:3.1}
E(z,c)= \Phi_1(z,c) - \I \Phi_2(z,c), \quad z\in\C.
\end{align}
A simple calculation, using the Lagrange identity and the fact that 
\begin{align}
 W_a(\Phi(\zeta^\ast,\cdot\,),\Phi(z,\cdot\,)) = 0, \quad \zeta,\, z\in\C,
\end{align} 
 shows that one has for all $\zeta$, $z\in\C$ 
\begin{align}\label{eqnREP}
 \frac{E(z,c) E(\zeta,c)^\ast - E(\zeta^*,c) E(z^\ast,c)^\ast}{2\I (\zeta^* -z)} = \int_a^c \spr{\Phi(\zeta,x)}{R(x) \Phi(z,x)} dx.
\end{align}
In particular, choosing $\zeta=z$ this shows that $E(\,\cdot\,,c)$ is a de Branges function, that is, $|E(z,c)|> |E(z^\ast,c)|$ for all $z$ in the open upper complex half-plane $\C^+$. 
Moreover, one observes that $E(\,\cdot\,,c)$ does not have any real zero $\lambda$, since otherwise both components of $\Phi(\lam,c)$ would vanish.

The de Branges space associated with the de Branges function $E(\,\cdot\,,c)$ will be denoted with $B(c)$ (see \cite[Section~19]{dBbook} for details). 
It consists of all entire functions $F$ such that the integral  
\begin{align}
 \int_\R \frac{|F(\lambda)|^2}{|E(\lambda,c)|^2} d\lambda
\end{align} 
is finite and such that $F/E(\,\cdot\,,c)$ as well as $F^\#/E(\,\cdot\,,c)$ are of bounded type in $\C^+$ (that is, they can be written as the quotient of bounded analytic functions) with non-negative mean type. 
Hereby, $F^\#$ denotes the entire function 
\begin{align}
 F^\#(z) = F(z^\ast)^\ast, \quad z\in\C,
\end{align}
 and the mean type of a function $N$ which is of bounded type in $\C^+$ is the number  
\begin{align}
 \limsup_{y\rightarrow\infty} \frac{\ln |N(\I y)|}{y} \in [-\infty,\infty).
\end{align}
Equipped with the inner product
\begin{align}
 \dbspr{F}{G}_{B(c)} = \frac{1}{\pi} \int_\R \frac{F(\lam)^* G(\lam)}{|E(\lam,c)|^2} d\lam, 
 \quad F,\, G\in B(c),
\end{align}
the space $B(c)$ turns into a reproducing kernel Hilbert space. 
In view of \cite[Theorem~19]{dBbook}) and Equation~\eqref{eqnREP}, the reproducing kernel $K(\,\cdot\,,\cdot\,,c)$ such that 
\begin{align}
 \dbspr{K(\zeta,\cdot\,,c)}{F}_{B(c)} = F(\zeta), \quad F\in B(c),
\end{align} 
for every $\zeta\in\C$ is simply given by
\begin{equation}\label{eqndBschrRepKer}
 K(\zeta,z,c) = \int_a^c \spr{\Phi(\zeta,x)}{R(x) \Phi(z,x)} dx, \quad \zeta,\,z\in\C.
\end{equation} 
In what follows, we will always identify $L^2((a,c);R(x)dx)$ with the subspace of functions in $L^2((a,b);R(x)dx)$ which vanish almost everywhere on $(c,b)$.

\begin{theorem}\label{thmdBschrBT}
For every $c\in (a,b)$ the transformation $f\mapsto\hat{f}$ defined by \eqref{eqnTrans} is unitary from $L^2((a,c);R(x)dx)$ onto $B(c)$. 
In particular,
\begin{equation}\label{eqnBn}
  B(c) = \big\lbrace  \hat{f} \,\big|\, f\in L^2((a,c);R(x)dx) \big\rbrace.
\end{equation}
\end{theorem}
 
\begin{proof}
The proof is (almost) literally the same as the ones of \cite[Theorem~3.2]{je}, \cite[Theorem~3.1]{egnt2}
\end{proof}

It remains to collect several useful properties of these de Branges spaces. 

\begin{corollary}\label{cordBprop}
 The de Branges spaces $B(c)$ have the following properties: 
 \begin{enumerate}[label=\emph{(}\roman*\emph{)}, leftmargin=*, widest=XX]
 \item\label{itdBpropi} The de Branges spaces $B(c)$ are isometrically embedded in $L^2(\R,d\rho)$; 
  \begin{equation}
   \int_\R |F(\lam)|^2 d\rho(\lam) = \|F\|^2_{B(c)}, \quad F\in B(c).
  \end{equation}
 \item\label{itdBpropii} The union of all de Branges spaces $B(c)$ is dense in $L^2(\R,d\rho)$; 
  \begin{equation}\label{eqndBschrdense}
   \overline{\bigcup_{c\in (a,b)} B(c)} = L^2(\R,d\rho).
  \end{equation}
 \item\label{itdBpropiii} The de Branges spaces $B(c)$ are strictly increasing; 
  \begin{equation}
   B(c_1)\subsetneq B(c_2) \quad\text{for }c_1 < c_2.
  \end{equation} 
 \item\label{itdBpropiv} The de Branges spaces $B(c)$ are continuous; 
  \begin{align}\label{eqndBcont}
   \overline{\bigcup_{x\in(a,c)} B(x)} = B(c) = \bigcap_{x\in(c,b)} B(x).
  \end{align}
 \end{enumerate}
\end{corollary}

\begin{proof}
Items~\emph{\ref{itdBpropi}} and~\emph{\ref{itdBpropii}} are an immediate consequence of Theorem~\ref{thmdBschrBT} and the fact that the transformation from~\eqref{eqnTrans} extends to a unitary map from $L^2((a,b);R(x)dx)$ onto $L^2(\R,d\rho)$. 
The inclusion in \emph{\ref{itdBpropiii}} follows from the similar fact that 
\begin{align*}
 L^2((a,c_1);R(x)dx) \subsetneq L^2((a,c_2);R(x)dx).
\end{align*} 
 In much the same manner, the remaining claim is due to  
 \begin{align*}
  \overline{\bigcup_{s\in(a,c)} L^2((a,s);R(x)dx)} = L^2((a,c);R(x)dx) = \bigcap_{s\in (c,b)} L^2((a,s);R(x)dx).
 \end{align*}
\end{proof}

\section{Inverse uniqueness results}\label{s4}

In order to state our inverse uniqueness results, let $\tau$ and $\tilde{\tau}$ be two differential expressions of the form~\eqref{eqnDiffExp}, both of them satisfying Hypothesis~\ref{hyp1}. 
Furthermore, we will denote with $S$ and $\tilde{S}$ two self-adjoint realizations with separated boundary conditions and suppose Hypothesis~\ref{hyp2} to hold. 
All of the quantities associated with $\tau$ are denoted as in the previous sections and, in obvious notation, the corresponding quantities for $\tilde{\tau}$ are equipped with an additional twiddle. 

\begin{theorem}\label{thmdBuniq}
Suppose that the function 
\begin{equation}\label{eqnquotE1E2}
   E(z,c)\tilde{E}(z,\tilde{c})^{-1}, \quad z\in\C^+
\end{equation}
is of bounded type for some $c\in (a,b)$ and $\tilde{c}\in (\tilde{a},\tilde{b})$. 
If $\rho = \tilde{\rho}$, then there is a locally absolutely continuous bijection $\eta$ from $(a,b)$ onto $(\tilde{a},\tilde{b})$ and a locally absolutely continuous function $\Gamma:(a,b)\rightarrow\R^{2\times 2}$ with $\det(\Gamma)=1$ such that 
\begin{align}\label{eqnUniqRel}
 \eta' \tilde{R}\circ\eta  & = \Gamma^\ast R\, \Gamma, & 
 \eta' \tilde{Q}\circ\eta  & = \Gamma^\ast Q\, \Gamma + \Gamma^\ast J\, \Gamma'.
\end{align}
Moreover, the operator $U$ from $L^2((\tilde{a},\tilde{b});\tilde{R}(x)dx)$ to $L^2((a,b);R(x)dx)$, given by 
\begin{align}
 U \tilde{f}(x) = \Gamma(x) \tilde{f}(\eta(x)), \quad x\in(a,b), 
\end{align}
is unitary with $S = U \tilde{S} U^{-1}$.
\end{theorem}

\begin{proof}
First of all note that because of the definition of de Branges spaces, the function in~\eqref{eqnquotE1E2} is of bounded type for all $c\in (a,b)$ and $\tilde{c}\in (\tilde{a},\tilde{b})$. 
Now fix some arbitrary point $x\in (a,b)$. 
Since for every $\tilde{x}\in (\tilde{a},\tilde{b})$, both spaces $B(x)$ and $\tilde{B}(\tilde{x})$ are isometrically embedded in $L^2(\R,d\rho)$, we infer from the subspace ordering theorem \cite[Theorem~35]{dBbook} that $B(x)$ is contained in $\tilde{B}(\tilde{x})$ or $\tilde{B}(\tilde{x})$ is contained in $B(x)$. 
 We claim that the infimum $\eta(x)$ of all $\tilde{x}\in (\tilde{a},\tilde{b})$ such that $B(x)\subseteq \tilde{B}(\tilde{x})$ lies in $(\tilde{a},\tilde{b})$.
 In fact, otherwise we either had $\tilde{B}(\tilde{x})\subseteq B(x)$ for all $\tilde{x}\in (\tilde{a},\tilde{b})$ or $B(x)\subseteq \tilde{B}(\tilde{x})$ for all $\tilde{x}\in (\tilde{a},\tilde{b})$.
 The first case would imply that $B(x)$ is dense in $L^2(\R,d\rho)$, which is not possible in view of Corollary~\ref{cordBprop}~\emph{\ref{itdBpropii}}.
 In the second case, this would mean that for every function $F\in B(x)$ and $\zeta\in\C$ we have
 \begin{align*}
 \begin{split}
  |F(\zeta)|^2 & \leq \left| \dbspr{\tilde{K}(\zeta,\cdot\,,\tilde{x})}{F}_{\tilde{B}(\tilde{x})} \right|^2
              \leq \|F\|_{\tilde{B}(\tilde{x})}^2 \dbspr{\tilde{K}(\zeta,\cdot\,,\tilde{x})}{\tilde{K}(\zeta,\cdot\,,\tilde{x})}_{\tilde{B}(\tilde{x})} \\
             & = \|F\|_{B(x)}^2 \tilde{K}(\zeta,\zeta,\tilde{x})
 \end{split}
 \end{align*}
 for every $\tilde{x}\in (\tilde{a},\tilde{b})$.
 But since $\tilde{K}(\zeta,\zeta,\tilde{x})\rightarrow0$ as $\tilde{x}\rightarrow \tilde{a}$ by~\eqref{eqndBschrRepKer}, this would imply $B(x)=\lbrace 0\rbrace$ in contradiction to Theorem~\ref{thmdBschrBT}.
 Now from~\eqref{eqndBcont} we infer that 
 \begin{align*}
   \tilde{B}(\eta(x)) = \overline{\bigcup_{\tilde{x}<\eta(x)} \tilde{B}(\tilde{x})} \subseteq B(x) \subseteq \bigcap_{\tilde{x}>\eta(x)} \tilde{B}(\tilde{x}) = \tilde{B}(\eta(x))
 \end{align*}
 and hence $B(x)=\tilde{B}(\eta(x))$, including norms. 
 In particular, from this we infer that there is a matrix $\Gamma(x)\in\R^{2\times 2}$ with $\det(\Gamma(x))=1$ such that
 \begin{align}\label{eqnGamma}
  \Phi(z,x) = \Gamma(x) \tilde{\Phi}(z,\eta(x)), \quad z\in\C,
 \end{align}
 in view of \cite[Theorem~I]{dB60}. 
  
 The function $\eta: (a,b) \rightarrow (\tilde{a},\tilde{b})$ defined above is strictly increasing and continuous by Corollary~\ref{cordBprop}.
 Moreover, since for every  $\zeta\in\C$ we have 
  \begin{align*}
    \tilde{K}(\zeta,\zeta,\eta(x)) = K(\zeta,\zeta,x) \rightarrow 0,
  \end{align*}
  as $x\rightarrow a$, we infer that $\eta(x)\rightarrow \tilde{a}$ as $x\rightarrow a$.
  Furthermore, the function $\eta$ even has to be a bijection because of~\eqref{eqndBschrdense}.  
 Next, since the reproducing kernels are given by~\eqref{eqndBschrRepKer}, we get for each $z\in\C$  
 \begin{align*}
  \int_{a}^{x} \spr{\Phi(z,s)}{R(s)\Phi(z,s)} ds = \int_{\tilde{a}}^{\eta(x)} \spr{\tilde{\Phi}(z,s)}{\tilde{R}(s) \tilde{\Phi}(z,s)} ds, \quad x\in (a,b).
 \end{align*}
 Thus  $\eta$ turns out to be locally absolutely continuous in view of \cite[Chapter~IX; Exercise~13]{na55} and \cite[Chapter~IX; Theorem~3.5]{na55} with 
 \begin{align}\label{eqnEtaP}
  \spr{\Phi(z,x)}{R(x)\Phi(z,x)} = \eta'(x) \spr{\tilde{\Phi}(z,\eta(x))}{\tilde{R}(\eta(x))\tilde{\Phi}(z,\eta(x))}
 \end{align}
 for almost all $x\in(a,b)$. 
 In order to prove that $\Gamma$ is locally absolutely continuous as well, fix some $z\in\C\backslash\R$ and note that from the Lagrange identity~\eqref{eqnLagrange} one gets 
 \begin{align}\label{eqnTilPhi}
 \begin{split}
  \tilde{\Phi}_1(z,\eta(x)) \tilde{\Phi}_2(z^\ast,\eta(x)) & - \tilde{\Phi}_2(z,\eta(x) \tilde{\Phi}_1(z^\ast,\eta(x)) \\
     & = 2\I\, \im(z) \int_{\tilde{a}}^{\eta(x)} \spr{\tilde{\Phi}(z,s)}{R(s)\tilde{\Phi}(z,s)}ds \not= 0
 \end{split}
 \end{align}
 for each $x\in(a,b)$. 
 Then from~\eqref{eqnGamma} we infer that 
 \begin{align*}
  \Gamma(x) =  \begin{pmatrix}  \Phi_1(z,x) & \Phi_1(z^\ast,x) \\ \Phi_2(z,x) & \Phi_2(z^\ast,x)  \end{pmatrix} 
                     \begin{pmatrix}  \tilde{\Phi}_1(z,\eta(x)) & \tilde{\Phi}_1(z^\ast,\eta(x)) \\ \tilde{\Phi}_2(z,\eta(x)) & \tilde{\Phi}_2(z^\ast,\eta(x))  \end{pmatrix}^{-1}, \quad x\in(a,b).     
 \end{align*} 
 Since the right-hand side 
 is locally absolutely continuous by \cite[Chapter~IX; Section~1]{na55}, so is $\Gamma$.  
 
  Next, we will show that the relations in~\eqref{eqnUniqRel} hold. 
  To this end, one plugs~\eqref{eqnGamma} into~\eqref{eqnEtaP} and notes that for $z\in\C\backslash\R$, the vectors $\tilde{\Phi}(z,\eta(x))$ and $\tilde{\Phi}(z^\ast,\eta(x))$ are linearly independent by~\eqref{eqnTilPhi}, which then shows that the first relation in~\eqref{eqnUniqRel} holds almost everywhere. 
  Subsequently, one differentiates~\eqref{eqnGamma} to obtain the second relation in~\eqref{eqnUniqRel} almost everywhere.

 Finally, consider the unitary operator $\mathcal{F}^{-1} \tilde{\mathcal{F}}$ which maps $L^2((\tilde{a},\tilde{b});\tilde{R}(x)dx)$ onto $L^2((a,b);R(x)dx)$ and satisfies $S = \mathcal{F}^{-1} \tilde{\mathcal{F}} \tilde{S} \tilde{\mathcal{F}}^{-1} \mathcal{F}$. 
 In order to identify this operator with $U$, one simply observes that for every $F\in L^2(\R;d\rho)$ with compact support 
 \begin{align*}
  \mathcal{F}^{-1} F(x) = \Gamma(x) \int_\R \tilde{\Phi}(\lambda,\eta(x)) F(\lambda) d\rho(\lambda) = \Gamma(x) \tilde{\mathcal{F}}^{-1} F(\eta(x)),
 \end{align*}
 for almost all $x\in(a,b)$, which finishes the proof.
\end{proof}

 In view of applications, the condition on~\eqref{eqnquotE1E2} being of bounded type is somewhat inconvenient. 
 However, let us mention that this condition always holds if the solutions $\Phi$ and $\tilde{\Phi}$ satisfy certain growth conditions. 
 Therefore recall that some entire function $F$ belongs to the Cartwright class if it is of finite exponential type and the logarithmic integral
 \begin{align}
  \int_\R \frac{\ln^+|F(\lam)|}{1+\lam^2}d\lam  
 \end{align}
is finite, where $\ln^+$ is the positive part of the natural logarithm. 
 Now if it is known that the functions $E(\,\cdot\,,c)$ and $\tilde{E}(\,\cdot\,,\tilde{c})$ belong to the Cartwright class for some $c\in(a,b)$ and $\tilde{c}\in (\tilde{a},\tilde{b})$, then a theorem of Krein~\cite[Theorem~6.17]{rosrov}, \cite[Section~16.1]{lev} guarantees that they are of bounded type in $\C^+$.
 Since the quotient of two functions of bounded type is of bounded type itself, this immediately implies that~\eqref{eqnquotE1E2} is of bounded type as well.
 In particular, we will see in the next section that it is always possible to choose the real entire solutions $\Phi$ and $\tilde{\Phi}$ such that this holds provided that the limit circle case prevails at the left endpoint.
 
 In general, one can not deduce more information of the Dirac operator from the spectral measure than claimed in Theorem~\ref{thmdBuniq}.  
 However, if one has further a priori information on the differential expression, then it is possible indeed to strengthen the results.  
 For example, in the case of Dirac operators with trivial weight matrices we get the following refinement of Theorem~\ref{thmdBuniq}.  
 
 \begin{corollary}\label{corIUI}
 Suppose that $R=\tilde{R}=I$ and that the function 
 \begin{equation}\label{eq:Ebt}
   E(z,c) \tilde{E}(z,\tilde{c})^{-1}, \quad z\in\C^+
 \end{equation}
 is of bounded type for some $c\in (a,b)$ and $\tilde{c}\in (\tilde{a},\tilde{b})$. 
 If $\rho = \tilde{\rho}$, then there is an $\eta_0\in\R$ and a locally absolutely continuous real-valued function $\varphi$ on $(a,b)$ such that 
 \begin{align}\label{eqnUniqRelPot}
  \tilde{Q}(\eta_0 + x) & = \E^{-\varphi(x)J} Q(x) \E^{\varphi(x) J}  - \varphi'(x) I
 \end{align}
 for almost all $x\in(a,b)$. 
 Moreover, the operator $U$ from $L^2((\tilde{a},\tilde{b});Idx)$ to $L^2((a,b);Idx)$ given by 
 \begin{align}
  U\tilde{f}(x) = \E^{\varphi(x) J} \tilde{f}(\eta_0 + x), \quad x\in(a,b),
 \end{align}
is unitary with $S = U \tilde{S} U^{-1}$.
\end{corollary}

\begin{proof}
 Employing Theorem~\ref{thmdBuniq}, we infer from our additional assumptions that 
 \begin{align*}
  \Gamma(x)^\ast \Gamma(x) = \eta'(x) I, \quad x\in (a,b). 
 \end{align*}
 Taking $\det(\Gamma)=1$ into account, we conclude that $\eta(x) = \eta_0 + x$ for some $\eta_0 \in\R$. 
 In particular,  this shows that $\Gamma(x)$ is unitary for each $x\in(a,b)$ and we may write 
 \begin{align*}
  \Gamma(x) = \begin{pmatrix} \cos\varphi(x) & -\sin\varphi(x) \\  \sin\varphi(x) & \cos\varphi(x) \end{pmatrix} = \E^{\varphi(x)J}, \quad x\in(a,b),
 \end{align*}
 for some real-valued function $\varphi$ on $(a,b)$. 
 Since the entries of $\Gamma$ are locally absolutely continuous, we infer that $\varphi$ may be chosen locally absolutely continuous as well.  
 Now evaluating the second equation in~\eqref{eqnUniqRel} establishes the claim. 
\end{proof}

If, in addition to the assumptions of Corollary~\ref{corIUI}, we furthermore require that the traces of the potentials $Q$ and $\tilde{Q}$ are equal almost everywhere (e.g.\ if both are normalized to the same constant), then we may conclude from~\eqref{eqnUniqRelPot} that the function $\varphi$ has to be a real constant. 
Under various additional a priori assumptions (for example, prescribing the boundary condition at a regular left endpoint), it is furthermore even possible to determine this constant (at least modulo $2\pi$).

As an application of Corollary \ref{corIUI} which is of particular physical interest we would like to single out 
the case of radial Dirac operators. Namely, we may consider differential expressions of the form 
\begin{equation}
\label{eq:radexprB}
 \tau f(x) = J f'(x) + \begin{pmatrix}q_{\rm sc}(x) & \frac{\kappa}{x}+q_{\rm am}(x) \\
\frac{\kappa}{x}+q_{\rm am}(x)  &-q_{\rm sc}(x)\end{pmatrix} f(x), \quad x\in(0,\infty), 
\end{equation}
with $q_{\rm sc}$, $q_{\rm am} \in L^1_{loc}([0,\infty),\R)$ and $\kappa\geq0$. In the case $\kappa=\frac{1}{2}$ we assume in addition 
\begin{align}
  \int_0^c \left(|q_{\rm sc}(x)| + |q_{\rm am}(x)|\right) |\log(x)|  dx < \infty 
\end{align}
for some $c\in(0,\infty)$. 
Here, $q_{\rm sc}$ and $q_{\rm am}$ are interpreted as a scalar potential and an anomalous magnetic moment, respectively (see \cite[Chapter~4]{th}). 
For simplicity we assume that the electrostatic potential is zero which can always be achieved by a suitable gauge transform. 
Moreover, the case $\kappa<0$ can be reduced to the case $\kappa>0$ by performing a Liouville transform with $\Gamma = J$. 
If $\kappa> 0$, then this differential expression is singular at the left endpoint as $\frac{\kappa}{x}$ is not integrable there and it is also singular at $\infty$ because the endpoint itself is not finite. Note that $\tau$ is always limit point at $\infty$ (see \cite[Theorem 8.6.1]{ls} and \cite[Theorem 6.8]{wdl}). 
We also refer to the monographs \cite{ls}, \cite{wdl} for background and also to \cite{th} for further information about Dirac operators and their applications.

Under the above assumptions, there is a particular real entire solution $\Phi$ associated with the differential expression~\eqref{eq:radexprB} which satisfies the spatial asymptotics  
\begin{align}\label{eq:phi0}
\Phi_1(z,x) & =\oo(x^\kappa), & \Phi_2(z,x) & =x^\kappa(1+\oo(1)),
\end{align}
as $x\to 0$ (see \cite{ser} in the case $\kappa\in\N_0$ and \cite[Section~8]{bekt} for the general case). 
In the case $0\le \kappa < \frac{1}{2}$, in which the differential expression is in the limit circle case at the left endpoint, this implies that we impose the boundary condition 
\begin{align}\label{eqnRDBC}
 \lim_{x\to0} x^\kappa f_1(x)=0.
\end{align}

\begin{theorem}\label{th:DiracUn}
Suppose that $\tau$, $\tilde{\tau}$ are differential expressions of the form~\eqref{eq:radexprB}  and that the boundary conditions of $S$, $\tilde{S}$ at the left endpoint are given by~\eqref{eqnRDBC} if the limit circle case prevails there. 
Furthermore, let $\Phi$, $\tilde{\Phi}$ be the real entire solutions which satisfy~\eqref{eq:phi0} and let $\rho$, $\tilde{\rho}$ be the corresponding spectral measures. 
If $\rho=\tilde{\rho}$, then $\kappa=\tilde{\kappa}$ and $q_{\rm sc}=\tilde{q}_{\rm sc}$, $q_{\rm am}=\tilde{q}_{\rm am}$ almost everywhere, as well as $S=\tilde{S}$. 
\end{theorem}

\begin{proof}
It follows from \cite[Lemma~8.3]{bekt} that the solutions $\Phi(\,\cdot\,,x)$ and $\tilde{\Phi}(\,\cdot\,,x)$ belong to the Cartwright class for every fixed $x>0$. 
Therefore, so do the functions $E(\,\cdot\,,x)$ and $\tilde{E}(\cdot,x)$ defined by \eqref{eq:3.1} and hence the function \eqref{eq:Ebt} is of bounded type. Thus we may apply Corollary~\ref{corIUI} and infer that $\eta_0=0$ because the left endpoints are the same. 
 Moreover, since the traces of the potential matrices are zero almost everywhere,  the function $\varphi$ is constant. 
 However, the particular form of the potential matrices immediately implies that $\sin\varphi=0$ (also taking into account that $\kappa$, $\tilde{\kappa}\geq0$). 
 Therefore, $\Gamma=\pm I$ and we infer from~\eqref{eqnUniqRelPot} that 
 \begin{align*}
  \begin{pmatrix}q_{\rm sc}(x) & \frac{\kappa}{x}+q_{\rm am}(x) \\
\frac{\kappa}{x}+q_{\rm am}(x)  &-q_{\rm sc}(x)\end{pmatrix} = \begin{pmatrix} \tilde{q}_{\rm sc}(x) & \frac{\tilde{\kappa}}{x}+ \tilde{q}_{\rm am}(x) \\
\frac{\tilde{\kappa}}{x}+ \tilde{q}_{\rm am}(x)  &- \tilde{q}_{\rm sc}(x)\end{pmatrix}
 \end{align*}
 for almost all $x\in(0,\infty)$, which finishes the proof.  
\end{proof}

Clearly, in view of \eqref{eqnSMdef}, we can rephrase this as the following result.

\begin{corollary}
Suppose that $\tau$, $\tilde{\tau}$ are differential expressions of the form~\eqref{eq:radexprB} and that the boundary conditions of $S$, $\tilde{S}$ at the left endpoint are given by~\eqref{eqnRDBC} if the limit circle case prevails there. 
Furthermore, let $\Phi$, $\tilde{\Phi}$ be real entire solutions which satisfy~\eqref{eq:phi0} and let $M$, $\tilde{M}$ be corresponding singular Weyl--Titchmarsh--Kodaira functions. 
If the function $M -\tilde{M}$ is entire, then $\kappa=\tilde{\kappa}$ and $q_{\rm sc}=\tilde{q}_{\rm sc}$, $q_{\rm am}=\tilde{q}_{\rm am}$ almost everywhere, as well as $S=\tilde{S}$. 
\end{corollary}

\begin{remark}
Note that various inverse spectral problems for radial Dirac operators in dimension three (i.e., the case $\kappa\in\Z$) inside the unit ball have been studied in \cite{ahm2, ahm3, ser} (see also the references therein) in the case $q_{\rm am}$, $q_{\rm sc}\in L^p(0,1)$, $p\in (1,\infty)$.
In particular, \cite{ahm2, ahm3, ser} provide a complete solution of the inverse spectral problems by one spectrum and norming constants and by two spectra in the special case $q_{\rm am}$, $q_{\rm sc}\in L^2(0,1)$.

The case $\kappa=0$ is of course classical  and we refer to the monograph \cite{ls} as well as the recent papers \cite{ahm, ma, ma2, mp, pu, sak2}.
For a local Borg--Marchenko result in the case $\kappa=0$ and $q_{\rm am}, q_{\rm sc}\in L^\infty(a,b)$ see \cite{clge}.
For related uniqueness results see \cite{gkm}.
\end{remark} 

Similarly, one can improve Theorem~\ref{thmdBuniq} in the case when the weight matrix is not constant but the potential vanishes identically. 

\begin{corollary}\label{corIUII}
 Suppose that $Q = \tilde{Q} = 0$ and that the function
 \begin{equation}
   E(z,c) \tilde{E}(z,\tilde{c})^{-1}, \quad z\in\C^+
 \end{equation}
 is of bounded type for some $c\in (a,b)$ and $\tilde{c}\in (\tilde{a},\tilde{b})$. 
 If $\rho = \tilde{\rho}$, then there is a locally absolutely continuous bijection $\eta$ from $(a,b)$ onto $(\tilde{a}, \tilde{b})$ and a matrix $\Gamma_0\in\R^{2\times 2}$ with $\det(\Gamma_0)=1$ such that 
 \begin{align}\label{eqnUniqRelWei}
  \eta'(x) \tilde{R}(\eta(x)) = \Gamma_0^\ast R(x)\, \Gamma_0
 \end{align} 
 for almost all $x\in(a,b)$. 
 Moreover, the operator $U$ from $L^2((\tilde{a},\tilde{b});\tilde{R}(x)dx)$ to $L^2((a,b);R(x)dx)$ given by 
 \begin{align}
  U \tilde{f}(x) = \Gamma_0 \tilde{f}(\eta(x)), \quad x\in(a,b), 
 \end{align}
 is unitary with $S = U \tilde{S} U^{-1}$.
\end{corollary}

\begin{proof}
 Employing Theorem~\ref{thmdBuniq}, we infer from our additional assumptions that 
 \begin{align*}
  \Gamma(x)^\ast J\, \Gamma'(x) = 0, \quad x\in(a,b). 
 \end{align*}
  Since $\det(\Gamma)=1$ and $J$ is invertible, we conclude that $\Gamma$ is constant which establishes the claim. 
 \end{proof}

If, in addition to the assumptions of Corollary~\ref{corIUII}, we furthermore require that the determinant of the weights $R$ and $\tilde{R}$ are equal to one almost everywhere, then we may conclude from~\eqref{eqnUniqRelWei} that $\eta$ is linear with gradient one. 

\begin{remark}
Let us mention that Corollary \ref{corIUII} was obtained in \cite[Proposition 4.9]{ww} under the assumption that ${\rm tr}(R) \in L^1(a,c)$ for some $c\in (a,b)$, that is, the left endpoint $x=a$ is in the limit circle case. However, the case of Hamiltonians $R$ that may vanish on sets of positive Lebesgue measure has also been studied in \cite{ww}.
\end{remark}

\section{Two-spectra uniqueness results}

In this section, we will show to which extent the (purely discrete) spectra of two self-adjoint realizations determine the underlying differential expression as well as boundary conditions.  
Therefore, one of the endpoints has necessarily to be in the limit circle case. 
If this is the case for the left endpoint, then Hypothesis~\ref{hyp2} holds for any self-adjoint realization with separated boundary conditions and we may even choose a particular real entire fundamental system of solutions.

\begin{proposition}\label{propLC}
 Suppose Hypothesis~\ref{hyp1}, that $\tau$ is in the limit circle case at $a$ and that $S$ is self-adjoint realization of $\tau$ with separated boundary conditions. 
 Then there is a real entire solution $\Phi$ of $(\tau-z)u=0$ which lies in $\dom{S}$ near $a$ as well as a real entire solution $\Theta$ of $(\tau-z)u=0$ such that 
 \begin{align}
  W_a(\Theta(z_1,\cdot\,),\Theta(z_2,\cdot\,)) & =  W_a(\Phi(z_1,\cdot\,),\Phi(z_2,\cdot\,))  = 0, &   W_a(\Theta(z_1,\cdot\,),\Phi(z_2,\cdot\,)) & = 1, 
 \end{align}
 for all $z_1$, $z_2\in\C$. 
 In this case, the corresponding Weyl--Titchmarsh--Kodaira function is a Herglotz--Nevanlinna function. 
\end{proposition}

This result can be proven along the lines of the corresponding result for Sturm--Liouville equations; see, e.g., \cite[App.~A]{kst2}, \cite[Sect.~5]{bekt}.
Note that hereby, the measure appearing in the Herglotz--Nevanlinna representation 
\begin{align}
 M(z) = \re(M(\I)) + m_c z + \int_\R \left(\frac{1}{\lambda-z} - \frac{\lambda}{1+\lambda^2}\right) d\rho(\lambda), \quad z\in\C\backslash\R
\end{align}
of the singular Weyl--Titchmarsh--Kodaira function is precisely the spectral measure of $S$ associated with the real entire solution $\Phi$ in view of~\eqref{eqnSMdef}. 
For the purpose of this section, we will now derive a growth restriction for these particular solutions. 
 
\begin{corollary}\label{corLPcartwright}
 Suppose Hypothesis~\ref{hyp1}, that $\tau$ is in the limit circle case at $a$ and that $S$ is a self-adjoint realization of $\tau$ with separated boundary conditions. 
 If $\Phi$ is a real entire solution of $(\tau-z)u=0$ as in Proposition~\ref{propLC}, then the entire functions $\Phi_1(\,\cdot\,,c)$ and $\Phi_2(\,\cdot\,,c)$ belong to the Cartwright class for every $c\in(a,b)$. 
\end{corollary}

\begin{proof}
 This follows essentially from an application of a result by Krein \cite{kr52}. 
 Alternatively, note that by Proposition~\ref{propLC} the fractions on the right-hand side of 
 \begin{align*}
  \frac{1}{\Phi_1(z,c)^2} = \frac{\Theta_1(z,c)}{\Phi_1(z,c)} \left( \frac{\Phi_2(z,c)}{\Phi_1(z,c)} - \frac{\Theta_2(z,c)}{\Theta_1(z,c)} \right), \quad z\in\C\backslash\R, 
 \end{align*}
 are Herglotz--Nevanlinna functions (up to the sign), which are bounded by 
 \begin{align*}
  C \frac{1+|z|^2}{|\im(z)|} \leq K \exp\left( L \frac{1+\sqrt{|z|}}{\sqrt[4]{|\im(z)|}} \right), \quad z\in\C\backslash\R,
 \end{align*}
 for some constants $C$, $K$, $L\in\R^+$.  
 Thus also
 \begin{align*}
  \left| \frac{1}{\Phi_1(z,c)^2} \right| \leq 2 K^2 \exp\left( 2L \frac{1+\sqrt{|z|}}{\sqrt[4]{|\im(z)|}} \right), \quad z\in\C\backslash\R,
 \end{align*}
 which guarantees that $\Phi_1(\,\cdot\,,c)$ belongs to the Cartwright class in view of Matsaev's theorem \cite[Theorem~26.4.4]{lev}). 
 In much the same manner one shows that the same holds true for $\Phi_2(\,\cdot\,,c)$. 
\end{proof}

In order to state our inverse uniqueness results, let $\tau$ and $\tilde{\tau}$ be two differential expressions of the form~\eqref{eqnDiffExp}, both of them satisfying Hypothesis~\ref{hyp1}. 
Again, all of the quantities associated with $\tau$ and $\tilde{\tau}$ will be denoted as in the preceding sections and distinguished with a twiddle. 

\begin{theorem}\label{thmIUTS}
 Suppose that $\tau$ and $\tilde{\tau}$ are in the limit circle case at the left endpoint and that $S$, $T$ as well as $\tilde{S}$, $\tilde{T}$ are two distinct self-adjoint realizations with purely discrete spectra and the same boundary condition at the right endpoint (if any). 
 If 
 \begin{align}\label{eqnTSequal}
  \sigma(S) = \sigma(\tilde{S}) \quad\text{and}\quad \sigma(T) = \sigma(\tilde{T}),
 \end{align}
 then there is a locally absolutely continuous bijection $\eta$ from $(a,b)$ onto $(\tilde{a},\tilde{b})$ and a locally absolutely continuous function $\Gamma:(a,b)\rightarrow\R^{2\times 2}$ with $\det(\Gamma)=1$ such that 
\begin{align}
 \eta' \tilde{R}\circ\eta  & = \Gamma^\ast R\, \Gamma, & 
 \eta' \tilde{Q}\circ\eta  & = \Gamma^\ast Q\, \Gamma + \Gamma^\ast J\, \Gamma'.
\end{align}
Moreover, the operator $U$ from $L^2((\tilde{a},\tilde{b});\tilde{R}(x)dx)$ to $L^2((a,b);R(x)dx)$ given by 
\begin{align}\label{eqnLiouvilleT}
 U \tilde{f}(x) = \Gamma(x) \tilde{f}(\eta(x)), \quad x\in(a,b), 
\end{align} 
is unitary with $S = U \tilde{S} U^{-1}$ and $T = U \tilde{T} U^{-1}$. 
\end{theorem}

\begin{proof}
Let $\Theta$, $\Phi$ be a real entire fundamental system of $(\tau-z)u=0$ as in Proposition~\ref{propLC} associated with the operators $S$. 
Due to our assumptions, the corresponding singular Weyl--Titchmarsh--Kodaira function $M$ is meromorphic and has poles precisely at the spectrum of $S$. 
 Furthermore, note that the value $h = M(\lambda)$ is independent of $\lambda\in\sigma(T)$ and conversely, each $\lambda\in\C$ for which $h = M(\lambda)$  is an eigenvalue of $T$. 
 Thus we may assume (upon replacing $\Theta$ by $\Theta + h \Phi$) that the set of zeros of $M$ is precisely the spectrum of $T$.  
 If $\tilde{\Theta}$, $\tilde{\Phi}$ is a similar real entire fundamental system of $(\tilde{\tau}-z)u=0$, then we infer from our assumptions~\eqref{eqnTSequal} and a theorem by Krein \cite[Theorem\ 27.2.1]{lev},  that $M=C^2 \tilde{M}$ for some positive constant $C$. 
 Moreover, upon replacing $\Phi$ with $C \Phi$ as well as $\Theta$ with $C^{-1}\Theta$, we may even assume that $M = \tilde{M}$. 
 From this we obtain $\rho = \tilde{\rho}$ and hence, except for the very last part, the claim follows from Theorem~\ref{thmdBuniq} (taking Corollary~\ref{corLPcartwright} into account).
 
 In order to show that $U$ maps $\tilde{T}$ to $T$, note that $U$ maps $\tilde{T}_{\mathrm{max}}$ given by 
\begin{align*}
\begin{split}
 \dom{\tilde{T}_{\mathrm{max}}} = \lbrace f \in L^2((\tilde{a},\tilde{b});\tilde{R}(x)dx) \,|\, f_1,\, f_2\in & AC_{\mathrm{loc}}(\tilde{a},\tilde{b}), \\ & \tilde{\tau}f\in L^2((\tilde{a},\tilde{b});\tilde{R}(x)dx) \rbrace, 
\end{split}
\end{align*}
 with $\tilde{T}_{\mathrm{max}} f = \tilde{\tau} f$ for $f\in \dom{\tilde{T}_{\mathrm{max}}}$, to the corresponding operator $T_{\mathrm{max}}$ given by 
\begin{align*}
\begin{split}
 \dom{T_{\mathrm{max}}} = \lbrace f \in L^2((a,b);R(x)dx) \,|\, f_1,\, f_2\in & AC_{\mathrm{loc}}(a,b), \\ & \tau f\in L^2((a,b);R(x)dx) \rbrace, 
\end{split}
\end{align*}
 with $T_{\mathrm{max}} f = \tau f$ for $f\in\dom{T_{\mathrm{max}}}$. 
 Thus, the operator $U \tilde{T} U^{-1}$ is a self-adjoint realization of $\tau$ with separated boundary conditions. 
 Furthermore, from the first part of the proof we know that the boundary condition at the right endpoint is the same as the one for $T$. 
 In order to show that the boundary condition at $a$ is the same as well, pick some $\lambda\in\sigma(T)$ and note that the function  
 \begin{align*}
  \Gamma(x) \tilde{\Theta}(\lambda,\eta(x)), \quad x\in(a,b),
 \end{align*}
 is a solution of $(\tau -\lambda)u=0$. 
 Moreover, it is an eigenfunction of $U \tilde{T} U^{-1}$ as well as of $T$, 
 which establishes the remaining claim. 
\end{proof}

Of course, as in Section~\ref{s4}, it is possible to improve on Theorem~\ref{thmIUTS} if one has further a priori information on the coefficients of the underlying differential expressions. 
For example, one obtains the following two analogues of Corollary~\ref{corIUI} and Corollary~\ref{corIUII} without changing their respective proof.

\begin{corollary}\label{corIUITS}
 Suppose that $R=\tilde{R}=I$, that $\tau$ and $\tilde{\tau}$ are in the limit circle case at the left endpoint and that $S$, $T$ as well as $\tilde{S}$, $\tilde{T}$ are two distinct self-adjoint realizations with discrete spectra and the same boundary condition at the right endpoint (if any). 
 If 
 \begin{align}
  \sigma(S) = \sigma(\tilde{S}) \quad\text{and}\quad \sigma(T) = \sigma(\tilde{T}),
 \end{align}
 then there is an $\eta_0\in\R$ and a locally absolutely continuous real-valued function $\varphi$ on $(a,b)$ such that 
 \begin{align}
  \tilde{Q}(\eta_0 + x) & = \E^{-\varphi(x)J} Q(x) \E^{\varphi(x) J}  - \varphi'(x) I
 \end{align}
 for almost all $x\in(a,b)$. 
 Moreover, the operator $U$ from $L^2((\tilde{a},\tilde{b});Idx)$ to $L^2((a,b);Idx)$ given by 
 \begin{align}
  U\tilde{f}(x) = \E^{\varphi(x) J} \tilde{f}(\eta_0 + x), \quad x\in(a,b),
 \end{align}
is unitary with $S = U \tilde{S} U^{-1}$ and $T = U \tilde{T} U^{-1}$.
\end{corollary}

Of course, as remarked after Corollary~\ref{corIUI}, one may infer that $\varphi$ is a constant upon requiring that the traces of $Q$ and $\tilde{Q}$ are equal almost everywhere. Also, one can apply this result to radial Dirac operators \eqref{eq:radexprB} in the case when $|\kappa|<\frac{1}{2}$.  

\begin{corollary}\label{corIUIITS}
 Suppose that $Q = \tilde{Q} = 0$, that $\tau$ and $\tilde{\tau}$ are in the limit circle case at the left endpoint and that $S$, $T$ as well as $\tilde{S}$, $\tilde{T}$ are two distinct self-adjoint realizations with discrete spectra and the same boundary condition at the right endpoint (if any). 
 If 
 \begin{align}
  \sigma(S) = \sigma(\tilde{S}) \quad\text{and}\quad \sigma(T) = \sigma(\tilde{T}),
 \end{align}
 then there is a locally absolutely continuous bijection $\eta$ from $(a,b)$ onto $(\tilde{a}, \tilde{b})$ and a matrix $\Gamma_0\in\R^{2\times 2}$ with $\det(\Gamma_0)=1$ such that 
 \begin{align}
  \eta'(x) \tilde{R}(\eta(x)) = \Gamma_0^\ast R(x)\, \Gamma_0
 \end{align} 
 for almost all $x\in(a,b)$. 
 Moreover, the operator $U$ from $L^2((\tilde{a},\tilde{b});\tilde{R}(x)dx)$ to $L^2((a,b);R(x)dx)$ given by 
 \begin{align}
  U \tilde{f}(x) = \Gamma_0 \tilde{f}(\eta(x)), \quad x\in(a,b), 
 \end{align}
 is unitary with $S = U \tilde{S} U^{-1}$ and $T = U \tilde{T} U^{-1}$.
\end{corollary}

Again, we may conclude that $\eta$ is linear with gradient one if we require the determinant of the weight matrices $R$ and $\tilde{R}$ to be equal to one almost everywhere. 
Finally, let us also mention that if the boundary conditions of $S$, $T$ and $\tilde{S}$, $\tilde{T}$ are known a priori,
then one may even say more about the matrix $\Gamma_0$ from evaluating~\eqref{eqnGamma} and its analog for $\Theta(z,x)$
 at $z=0$ (note that $\tau$ and $\tilde{\tau}$ are necessarily regular at the left endpoint in this case).

\bigskip
\noindent
{\bf Acknowledgments.}
We thank Fritz Gesztesy and Alexander Sakhnovich for hints with respect to the literature.
J.E.\ and G.T.\ gratefully acknowledge the stimulating atmosphere at the {\em Institut Mittag-Leffler} during spring 2013 where parts of this paper were written during the international research program on {\em Inverse Problems and Applications}.

\end{document}